\documentclass[11pt,reqno]{amsart}
\usepackage[margin=3.1 cm]{geometry}
\usepackage{cite}
\usepackage{xcolor}
\usepackage{tikz-cd}
\usepackage{subcaption}
\usepackage{float}
\usepackage{bm}
\usepackage{mathtools}
\usepackage{graphicx}
\usepackage{enumerate}
\usepackage{amsfonts,amssymb,verbatim,amsmath,amsthm,latexsym,textcomp,amscd,hyperref}
\usepackage{latexsym,amsfonts,amssymb,epsfig,verbatim}
\usepackage{graphicx}
\usepackage{color}
\usepackage{url}
\usepackage{pgfplots}
\usepackage{tikz}
\usepackage{pdfpages}
\graphicspath{ {./images/} }
\newtheorem{theorem}{Theorem}[section]
\newtheorem*{theorem*}{Theorem}
\newtheorem*{lemma*}{Lemma}
\newtheorem*{proposition*}{Proposition}
\newtheorem*{corollary*}{Corollary}
\newtheorem{lemma}[theorem]{Lemma}
\newtheorem{defn}[theorem]{Definition}
\newtheorem{prop}[theorem]{Proposition}
\newtheorem{cor}[theorem]{Corollary}

\theoremstyle{definition}
\newtheorem{remark}[theorem]{Remark}

\def\dl{\delta}

\def\ep{\epsilon}

\def\sse{\subseteq}

\def\bt{\beta}
\def\al{\alpha}

\def\pa{\partial}
\def\map{\rightarrow}
\def\isom{\cong}

\newcommand\RED{\textcolor{red}}

\def\lgl{\langle}
\def\rgl{\rangle}

\def\smallskip{\vspace\smallskipamount}
\def\medskip{\vspace\medskipamount}
\def\bigskip{\vspace\bigskipamount}

\def\A{\mathcal{A}}

\def\R{\mathbb {R}}

\def\N{\mathbb {N}}

\def\Z{\mathbb {Z}}

\begin{document}
\title[Weakly malnormal quasiconvex subgroups]{On the existence of weakly malnormal quasiconvex subgroups of hyperbolic groups}
\author{Rakesh Halder and Pranab Sardar}

\address{School of Mathematics, Tata Institute of Fundamental Research (TIFR) Mumbai, Dr. Homi Bhaba Road, Colaba, Mumbai, Maharastra 400005, India}
\email{rhalder.math@gmail.com}

\address{Indian Institute of Science Education and Research (IISER) Mohali,Knowledge City,  Sector 81, S.A.S. Nagar, Punjab 140306, India}
\email{psardar@iisermohali.ac.in}


	\subjclass[2020]{20F65, 20F67}
	\keywords{Hyperbolic metric spaces (groups), (weakly) malnormal subgroups, commensurated subgroup.}

	\maketitle
\begin{abstract}In this short note, we prove the existence of weakly malnormal, virtually free, quasiconvex
subgroups in any nonelementary hyperbolic group. This extends a result of Ilya Kapovich
appearing in \cite{kapovich-nonqc}, where he proved the existence of malnormal quasiconvex subgroups
in torsion-free hyperbolic groups.
\end{abstract}

\setcounter{secnumdepth}{3}  
\setcounter{tocdepth}{3}     

\tableofcontents

\section{Introduction}\label{introduction}

Hyperbolic groups, introduced by M. Gromov in the $1980$s (\cite{gromov-hypgps}), are one of the widely
studied objects in Geometric Group Theory. They are obtained as generalization to geometrically finite
Kleinian groups acting on classical hyperbolic $n$-spaces. Quasiconvex subgroups of hyperbolic groups
are subgroups which themselves behave in the same way, i.e. geometrically finite or convex cocompact.
It is well known that a subgroup $H$ of a hyperbolic group $G$ is quasiconvex if and only if it is undistorted
or quasiisometrically embedded in $G$ (\cite[Lemma $3.6$, III.H]{bridson-haefliger}). In any hyperbolic groups there are plenty of free quasiconvex
subgroups. See for instance, \cite[$5.3.E$]{gromov-hypgps}, \cite[Lemma $3.20$, III.H]{bridson-haefliger}. However, results demonstrating constructions of `interesting' examples of quasiconvex
or nonquasiconvex subgroups of hyperbolic groups are very limited. In \cite{kapovich-nonqc}
Ilya Kapovich proved a nonquasiconvex embedding theorem. However, in the proof of the theorem in that
paper the following auxiliary result was established.

\begin{theorem}\textup{(\cite[Theorem $6.7$]{kapovich-nonqc})}
Suppose $G$ is a torsion free, nonelementary hyperbolic group. Then, in $G$ there is a malnormal
quasiconvex subgroup $H$ which is a free group of rank at least $2$.
\end{theorem}

However, although the corresponding result for general hyperbolic groups, i.e. those which need not be
torsion free is missing in the literature. We believe that such a result is known to many experts
but no written account of the same was found. Hence, in this paper, we prove the following
result in this connection.\smallskip

{\bf Theorem A} \textup{(Theorem \ref{main malnormal thm})}\label{thm-main-intro} {\em
Let $G$ be a nonelementary hyperbolic group. Then there is a  weakly malnormal and quasiconvex subgroup of the form $Comm_G(H_1)=H_1\times A$ in $G$ where $H_1$ is a free subgroup of rank at least $2$ and $A$ is a finite subgroup. 
}\smallskip




As an application of the above theorem one immediately gets the following result
using the Bestvina--Feighn combination theorem (\cite[Theorem $1.2$]{BF-Adn}). This result generalizes the main result of \cite{kapovich-nonqc}.\smallskip

{\bf Theorem B} \textup{(Theorem \ref{main application})}
{\em Suppose $G_1$ is any nonelementary hyperbolic group. Then there is a hyperbolic group $G$
and an injective homomorphism $\phi: G_1\map G$ such that $\phi(G_1)$ is not quasiconvex
in $G$.}

\smallskip

{\bf Remark}: In our proof of Theorem A, a substantial part of the proof idea is borrowed from \cite{kapovich-nonqc}. Both approaches share the following key observation: given a free group $\mathbb F_n$ of rank $n\ge2$, and a finite collection of infinite, infinite index quasiconvex subgroups $H_1,H_2,\cdots,H_l$, there exists a subgroup $F<\mathbb F_n$ of rank $2$ such that no element of $F$ is conjugate in $\mathbb F_n$ to an element of $H_1\cup\cdots\cup H_l$. See \cite[Theorem 5.16]{kapovich-nonqc} and Proposition \ref{pingpong prop}.

Moreover, the subgroup $F$ in I. Kapovich's construction is malnormal in $\mathbb F_n$, which took him to work a lot \cite[Section $5$]{kapovich-nonqc}. However, one can easily find such a malnormal subgroup using {\em graded small-cancellation} of Wise (\cite{wise-mal-1,wise-mal-2}). Finally, I. Kapovich showed that this malnormal subgroup remains malnormal in the ambient group $G$, which was easier as the group $G$ was torsion-free.

In our case, the generators of $F$ come from a very restricted class (Proposition \ref{pingpong prop}) as required later in our proof. We achieve such $F$ by studying the \emph{coned-off graphs} of Cayley graphs of hyperbolic groups which need not be free (see Proposition \ref{pingpong criteria}). Therefore, the final approach we adopt differs from that of I. Kapovich (See Section \ref{sec-main thm}).


\smallskip
{\bf Organization of the paper}:
In Section \ref{sec-gen on mal} we recall some standard facts about malnormality, commensurability
etc. and set some notation for future use. Section \ref{sec-pre on hyp geo} contains
basics of hyperbolic geometry. Section \ref{sec-qc subgroups} is devoted to properties of quasiconvex
subgroups of hyperbolic groups. This is rather long and it contains most of the
technical results of the paper. Finally we prove our main theorems in Section \ref{sec-main thm}.

\section{Generalities on malnormality and related notions}\label{sec-gen on mal}

Suppose $G$ is an infinite group with an infinite subgroup $H$.

\begin{defn}
We say that $H$ is {\em weakly malnormal} or {\em almost malnormal} in $G$
if for all $g\in G\setminus H$, the intersection $H\cap H^g$ is
finite.
\end{defn}
This notion was introduced by B. Baumslag (\cite{BB-malnormal}).
This is sort of opposite or antithetic to being normal. More generally, the
following concept is relevant in this context.

\begin{defn}
The {\em commensurator} of $H$ in $G$ is the subgroup of $G$ defined by
$$Comm_G(H)=\{ g\in G: [H: H\cap H^g] \, \mbox{and} \, [H^g: H\cap H^g]\, \mbox{ are finite}\}$$
\end{defn}
We note that in general $H<Comm_G(H)$ and $H\neq Comm_G(H)$ implies that $H$ is not weakly
malnormal in $G$.

We define the following subsets to analyze weak malnormality:

$$\mathcal I_G(H)=\{ g\in G: H\cap H^g \, \mbox{ is infinite}\}.$$

$$ \mathcal F_G(H)=\{g\in G\setminus H: H\cap H^g \, \mbox{ is finite}\}.$$

Clearly the three sets $Comm_G(H)$, $\mathcal I_G(H)$ and $\mathcal F_G(H)$
satisfy the following properties for an infinite subgroup $H$.

\begin{lemma}\label{lem-properties of I(H)}
\begin{enumerate}
\item The sets $Comm_G(H) \subset \mathcal I_G(H)$ and $\mathcal F_G(H)\cap \mathcal I_G(H)=\emptyset$.
\item $H$ is weakly malnormal in $G$ if and only if $\mathcal F_G(H)=G\setminus H$, and in which
case $Comm_G(H)=H=\mathcal I_G(H)$.
\item The natural action of $H$ by conjugation on $G$, keeps these sets stable.
\end{enumerate}
\end{lemma}

\begin{defn}
We shall say that a subgroup $H$ of a group $G$ is {\em self-commensurated} if $H=Comm_G(H)$.
\end{defn}

It follows that weakly malnormal subgroups are self-commensurated. Also for any $H<G$,
$Comm_G(H)$ is self-commensurated if $H$ has finite index in $Comm_G(H)$. Thus, later in this paper,
in search for weakly malnormal subgroups we always ensure first that a subgroup is self-commensurated.

Most of what follows can be thought of as a generalization of the following theorem of
Ilya Kapovich.

\begin{theorem}\textup{(\cite[Theorem $6.7$]{kapovich-nonqc})}\label{malnormal in free}
Suppose $H<G$ are nonelementary, finitely generated free groups. Then there is a nonelementary,
finitely generated (free) subgroup $K<H$ such that $K$ is malnormal in $G$.
\end{theorem}

{\bf Question.} {\em Suppose $H<G$ are infinite groups which are not necessarily torsion free.
Are there subgroups $K<H$ such that $Comm_G(K)$ is weakly malnormal in $G$?}\smallskip

One in general may restrict to a family of group pairs $H<G$ of ones choice and ask the
same question and also may require the subgroup $K$ to have certain properties. In our
case we take a nonelementary quasiconvex subgroup $H$ of a hyperbolic group $G$ and also require
$K$ to be quasiconvex and show in this paper that the answer is positive in this case.

In the rest of this section we pick up some other results that will be needed later in
the paper.
\subsection{Some other algebraic results}
Following are two other easy lemmas which are used later in the paper.
\begin{lemma}\label{malnormal transitive}
Suppose we have three groups $K<H<G$. If $H$ is (weakly) malnormal in $G$ and $K$ is
(weakly) malnormal in $H$ then $K$ is (weakly) malnormal in $G$.
\end{lemma}

\begin{lemma}\label{lemma: commensurator}
Suppose $G$ is any group and $H,K$ any two subgroups such that the index of $H\cap K$
is finite in both $H$ and $K$. Then $Comm_G(H)=Comm_G(K)$.
\end{lemma}

However, the following old result due to G. Baumslag and T. Taylor will be crucially used in the
proof of our main theorem.

\begin{theorem}\textup{(\cite[Proposition 1]{BT})}\label{finite gp inj}
	Suppose $K$ is a finite subgroup of $Aut(\mathbb F_n)$ where $n\ge2$. Then the restriction of the natural map $Aut(\mathbb F_n)\map GL_n(\Z)$ to $K$ is injective.
\end{theorem}

\section{Prerequisites from hyperbolic geometry}\label{sec-pre on hyp geo}

\subsection{Some definitions from coarse geometry}
For various notions and results in this section one is refered to \cite{gromov-hypgps} or
 \cite[Chapters I.1, III.H]{bridson-haefliger}. All our metric spaces are geodesic metric spaces
 and most of the time we work with graphs. Graphs are always assumed to be connected and the edge
 lengths are assumed to be $1$ so that these are geodesic metric spaces.

 A {\em geodesic segment} joining $x$ and $y$ in a metric space $X$ is denoted by $[x,y]_X$
 (or simply $[x,y]$ when $X$ is understood) and this is the image of an isometric embedding
 $\alpha: [0, d(x,y)]\map X$ with $\alpha(0)=x$ and $\alpha(d(x,y))=y$. For two subsets
 $A,B$ of a metric space $X$, the {\em Hausdorff distance} between $A$ and $B$ is
 $Hd(A,B):=inf~\{r\geq 0:A\sse N_r(B),B\sse N_r(A)\}$.
 A metric space $X$ is said to be {\em proper} if closed and bounded subsets of $X$ are compact.

Suppose $f:(X,d_X)\map(Y,d_Y)$ is a map between two metric spaces. Let $k\ge1$ and $\ep\ge0$. We say that $f$ is a
$(k,\epsilon)$-{\em quasiisometric (qi) embedding} (or just a qi embedding when the constants are
not important) if $$\frac{1}{k}d_X(x,x')-\ep\le d_Y(f(x),f(x'))\le kd_X(x,x')+\ep.$$
When the map $f$ is, moreover, an inclusion (e.g. when $X$ is a subgraph of $Y$) then we say that
$X$ is qi embedded in $Y$. A quasiisometric embedding $f:X\map Y$ is called a {\em quasiisometry}
if there is $R\geq 0$ such that $N_R(f(X))=Y$.
A map $\al:I\sse\R\map X$ is said to be $k$-{\em quasigeodesic} if it is $k$-qi embedding
where $I$ is an interval in $\mathbb R$. We note that in this case $\al$ is a geodesic if
$k=1, \epsilon =0$. When, moreover, $I$ is of the form $(-\infty, a]$ or $[a,\infty)$ for
some $a\in \mathbb R$ then $\al$ is called a {\em quasigeodesic ray} and when $I=\mathbb R$ then
it is called a {\em quasigeodesic line}.

\subsection{Hyperbolic metric spaces and hyperbolic groups}
We assume that the reader is familiar with hyperbolic spaces (and groups) and their boundaries
(see \cite{gromov-hypgps}, \cite{abc}, \cite[Chapter III.H]{bridson-haefliger} and \cite{GhH}).
In this section we briefly mention some of the results that we will need later in the paper.
Proofs of most of the results are omitted and references are given.

We recall that a geodesic metric space $X$ is  $\dl$-hyperbolic for some $\dl\geq 0$
(or just hyperbolic when the constant is not important) if the geodesic triangles are $\dl$-slim.
A finitely generated group $G$ is hyperbolic if one (any) of its Cayley graph with respect to
some (any) finite generating set is hyperbolic. In most cases, a choice of a generating set is
assumed and then the corresponding Cayley graph is denoted by $\Gamma_G$ or simply $\Gamma$.

\subsubsection{Boundary at infinity}
Suppose $X$ is a hyperbolic metric space.
Two (quasi)geodesic rays $\al$ and $\bt$ in $X$  are said to be {\em asymptotic} if their
Hausdorff distance is finite. This defines an equivalence relation on the set of (quasi)geodesic rays.
The set of all equivalence classes of such rays is called the (quasi)geodesic boundary of $X$
and is denote by $\pa X$ (or $\pa_q X$ for the case of quasigeodesic boundary).

For a (quasi)geodesic ray $\al$, its equivalence class in $\pa X$ or $\pa_q X$ is denoted
by $\al(\infty)$. The natural map $\pa X\map \pa_q X$ is injective and it is a bijection for
proper hyperbolic spaces.\smallskip

{\bf Convention.} {\em From now, we shall abuse notation and use just $\pa X$ in stead of
$\pa_q X$ even in case of nonproper spaces, e.g. coned-off graphs of Cayley graphs.}\smallskip

However, in general, one can define a topology on $\bar{X}=X\cup \pa X$. See \cite[Chapter III.H, p. $432$]{bridson-haefliger}, \cite[p. $49$]{abc}. As we do not need
this explicitly anywhere in the paper we skip its definition and instead mention
various features of it through many lemmas and propositions. For instance we will make use of
the following lemma repeatedly in the paper`.

\begin{lemma}
Suppose we have an isometry $f:X\map X$ of hyperbolic metric space $X$. Then $f$ extends to a
homeomorphism $\bar{f}: \bar{ X}\map \bar{X}$.

This defines an action of $Isom(X)$ $-$ the group of isometries of $X$, on $\bar{X}$.
\end{lemma}
 The restriction of $\bar{X}$ in the above lemma to $\pa X$ is denoted by $\pa f$.
As quasiisometries are not continuous in general one cannot expect the analogue of the above
lemma to hold in that generality. However, each quasiisometry between hyperbolic metric
spaces induce a homeomorphism on the boundary. Using this one shows that boundary of a
hyperbolic group is well defined up to homeomorphism. Hence, for a hyperbolic group $G$,
we shall denote by $\pa G$ the boundary of any Cayley graph $\Gamma_G$ of $G$.
Moreover, the natural $G$-action on $\Gamma_G$ by isometries induces an action on
$\pa G$ through homeomorphisms.

\begin{defn}
Suppose $X$ is a hyperbolic metric space and $f:X\map X$ is an isometry.
For any $x\in X$ let $\mathcal O_x:\mathbb Z \map X$ be the map $n\mapsto f^n(x)$.

(1) We say that $f$ is an {\em elliptic} isometry of $X$ if the image of the map $\mathcal O_x$ has finite
diameter in $X$ for some (any) $x\in X$..

(2) We say that $f$ is a {\em loxodromic} isometry if for some (any) $x\in X$, the map
$\mathcal O_x$ is qi embedding of $\mathbb Z$ into $X$.
\end{defn}

\begin{lemma}{\em (Characteristics of loxodromic isometries)}\label{char loxo isom}
(a) Suppose $f$ is a loxodromic isometry of a hyperbolic metric space $X$. Let $x\in X$.
Then the following hold.

(1) $f$ has exactly two  fixed points in $\pa X$, namely the limits $\lim_{n\map \infty} f^{\pm n}(x)$
{\em (which we denote by $f^{\pm \infty}$ respectively)}.

(2) Given $\xi \in \pa X$, if $\xi\neq f^{-\infty}$ then $\lim_{n\map \infty} f^n.\xi= f^{\infty}$.

\smallskip
(b) Suppose $f$ is a nonelliptic isometry of a hyperbolic metric space $X$ such that $f$ has precisely
two fixed points in $\pa X$. Then $f$ is a loxodromic isometry.
\end{lemma}

\begin{defn}
Suppose $f,g$ two loxodromic isometries on a hyperbolic metric space $X$. We say that they are
{\em independent} loxodromic isometries if they have no common fixed points in $\pa X$.
\end{defn}

\begin{lemma}\label{lem-finding free qc}
Suppose $X$ is a hyperbolic metric space which may not be proper, and suppose $f,g$ are two independent loxodromic isometries
on $X$. Then there is $m\in\N$ such that the group $H=\lgl f^m,g^m\rgl \simeq \mathbb F_2$. Moreover, any orbit map
$H \map X$ is quasiisometric embedding.
\end{lemma}

We note that a hyperbolic group $G$ is said to be {\em nonelementary} if $G$ is not virtually cyclic.
\begin{lemma}\textup{(\cite[Lemma $3.10$]{bridson-haefliger}, \cite[Corollary $8.2.G$]{gromov-hypgps})}\label{lem-infinite ele loxo}
In a hyperbolic group $G$, any infinite order element $g$ acts on a Cayley graph of $G$ as a loxodromic isometry.

In a nonelementary hyperbolic group there are infinite order elements (which we call independent infinite order elements)
that induce two independent loxodromic isometries on a Cayley graph of $G$.
\end{lemma}
The following lemma is quite standard. But we could not find an explicit reference. Therefore, we
give a sketch of proof for it.

\begin{lemma}\label{lemma lox}
Suppose $f$ is an isometry of a hyperbolic metric space $X$ such that $f$ is not an elliptic
isometry and $f$ has at least two fixed points in $\partial X$. Then $f$ is loxodromic.
\end{lemma}
{\em Sketch of proof:} Let $\xi, \xi'\in \pa X$ be two points fixed by $f$.
We look at the union of all (uniform) quasigeodesics lines (see \cite[Lemma 2.4]{pranab-mahan})
in $X$ connecting $\xi, \xi'$. This is a quasiconvex set, say $A$, in $X$.
A small neighbourhood of $A$ in $X$ is path connected and with the induced path metric
from $X$ it is qi embedded in $X$ (see \cite[Lemma 1.97]{ps-kap}). Let $Y$ denote this
small neighbourhood of $A$ in $X$.
One then sees that the infinite cyclic group $\lgl f\rgl$ acts metrically properly and
coboundedly on $Y$ and completes the proof by using a version of Švarc--Milnor lemma
 (see \cite[Lemma 1.32]{ps-kap}).
\qed

\subsubsection{Limit sets}
\begin{defn}
Suppose $X$ is a hyperbolic metric space and $A\subset X$.
A point $\xi\in \pa X$ is called a {\em limit point} of $A$ if there is a sequence of points
$\{x_n\}$ in $A$ such that $\lim_{n\map \infty} x_n =\xi$ (in the topology of $\bar{X}$).

The limit set of $A$ in $\pa X$ is the set of all limit points of $A$.

If $\mathcal G$ is a group of isometries on $X$ then the limit set of $\mathcal G$
is the limit set of any orbit $\mathcal G.x$ where $x\in X$.

If $G$ is a hyperbolic group and $A\subset G$, $\Lambda_G(A)$ or simply $\Lambda(A)$
denotes the limit set of $A$ in $\pa G$ where $\pa G$ is understood to be the boundary
of a Cayley graph of $G$.
\end{defn}
It is not hard to see that $\Lambda_G(X)$ in the above definition is independent of $x\in X$.
The following are some common properties of limit sets of groups of isometries.

The following Lemma \ref{limit set basic} $(1)$ and $(2)$ are straightforward consequences of the version of Arzelà--Ascoli theorem as in \cite[Lemma $3.10$, I.3]{bridson-haefliger}; see the proof of \cite[Lemma $3.1$, III.H]{bridson-haefliger} for instance. Moreover, $(3)$ easily follows from $(2)$.

\begin{lemma}  \label{limit set basic}
(1) If $X$ is a proper hyperbolic metric space then for any set $A\subset X$,
$\Lambda_X(A)\neq \emptyset$ if and only if $A$ has infinite diameter.

In particular, for a hyperbolic group $G$, and a subset $A\subset G$,
$\Lambda_G(A)\neq \emptyset $ if and only if $A$ is infinite.

(2) Suppose $X$ is any hyperbolic metric space and $A,B$ are two subsets of $X$.
If $Hd(A,B)<\infty$ then $\Lambda_X(A)=\Lambda_X(B)$.

(3) Suppose $G$ is a hyperbolic group. Then the following hold:

(i) If $K<H$ are some subgroups of $G$ then $[H:K]<\infty$ implies
$\Lambda(K)=\Lambda(H)$.

(ii) For any $H<G$ and $x\in G$, $x\Lambda(H)=\Lambda(xH)=\Lambda(H^x)$ where $H^x$ denotes the conjugate of $H$ by $x$.
\end{lemma}

\begin{lemma}\textup{(\cite[Theorem $5.1$]{coornaert-mea})}\label{limit set minimal}
(1) Suppose $X$ is a hyperbolic metric space, and $\mathcal G$ is a group acting
properly discontinuously by isometries on $X$. Then $\Lambda_X(\mathcal G)$ is the
minimal, closed, $\mathcal G$-invariant subset of $\pa X$.

(2) Suppose $G$ is a hyperbolic group and $H$ is an infinite subgroup of $G$.
Then $\Lambda_G(H)$ is the minimal, closed, $H$-invariant subset of $\pa G$.

In particular, any $G$-orbit on $\pa G$ is dense in $\pa G$.
\end{lemma}

\section{Quasiconvex subgroups of hyperbolic groups}\label{sec-qc subgroups}
In this section we recall various results on quasiconvex subgroups of hyperbolic groups.
We recall that in a hyperbolic metric space $X$, a subset $A$ is called $k$-{\em quasiconvex}
for some $k\geq 0$ if for all $x,y\in A$ and all geodesic
$\gamma$ in $X$ joining $x,y$ we have $\gamma \subset N_k(A)$. When the constant $k$ is not
important, we say that $A$ is quasiconvex in $X$.
A subgroup $H$
of a hyperbolic group $G$ is said to be quasiconvex if it is quasiconvex in some (any)
Cayley graph of $G$. The following proposition summarizes some of the basic properties
of quasiconvex subgroups of hyperbolic groups which will be needed later.

\begin{prop}\label{qc prop}
Suppose $G$ is a hyperbolic group.
\begin{enumerate}
\item \textup{(\cite[Lemma $3.6$]{bridson-haefliger})} A subgroup $H$ is quasiconvex if and only if $H$ is a finitely generated, qi embedded subgroup of $G$.
In this case, $H$ is also hyperbolic. Quasiconvexity is independent of generating sets of $H$
and $G$.
\item Finite subgroups of $G$ are quasiconvex in $G$. Also, finite index subgroups of $G$
are quasiconvex in $G$.
\item Any cyclic subgroup of $G$ is qi embedded and hence quasiconvex.
If $g\in G$ is of infinite order then $g^{\pm \infty}=\lim_{n\map \pm \infty}g^n$ are the
unique fixed points of $g$ in $\pa G$ and $\Lambda_G(\lgl g\rgl)=\{g^{\pm \infty}\}$.
\item \textup{(\cite{GMRS})} If $H$ is quasiconvex in $G$ then $H$ has finite index in $Comm_G(H)$.
\item If $g,h\in G$ are two infinite order elements then either $\Lambda_G(\lgl g\rgl)=\Lambda_G(\lgl h\rgl)$
or $\Lambda_G(\lgl g\rgl)\cap \Lambda_G(\lgl h\rgl)=\emptyset$.

\item If $g,h \in G$ are two infinite order elements with different fixed points in $\pa G$,
then there is $n\in \N$ such that $\lgl g^n, h^n\rgl\isom \mathbb F_2$ is a qi embedded
subgroup of $G$.

If $G$ is nonelementary there are such pairs of infinite order elements.
\item Suppose $K<H<G$. If $H$ is quasiconvex in $G$ and $K$ is quasiconvex in $H$ then $K$
is quasiconvex in $G$.

If $K$ has finite index in $H$ then $K$ is quasiconvex in $G$ if and only if $H$ is quasiconvex in $G$.
\item \textup{(cf. \cite[Proposition $3$]{short}, \cite[Lemma $2.6$]{GMRS})}
Suppose $H_1, H_2$ are two quasiconvex subgroup of $G$. Then so is $H_1\cap H_2$.

Moreover, in this case, $\Lambda_G(H_1)\cap \Lambda_G(H_2)=\Lambda_G(H_1\cap H_2)$.
\item \textup{(\cite[Theorem $1.7$, III.H]{bridson-haefliger})}
Given $k\geq 0$ and $\delta \geq 0$ there is $D\geq 0$ such that the following holds:

Suppose $H$ is a $k$-quasiconvex subgroup of a $\delta$-hyperbolic group $G$ and
$\xi_1, \xi_2\in \Lambda_G(H)$ are two distinct points. Then for any geodesic line $\gamma$
in $\Gamma_G$ joining $\xi_1, \xi_2$, $\gamma \subset N_D(H)$.

\item \textup{(\cite{GMRS}[Lemma $2.9$])}
Suppose $G$ is a hyperbolic group and $H$ is a quasiconvex subgroup of $G$ such that
$\Lambda_G(H)=\pa G$. Then $H$ has finite index in $G$.
\end{enumerate}
\end{prop}

\begin{lemma}\label{prop-finite-union-nowheredense}
Let $G$ be a nonelementary hyperbolic group. Let $H$ be an infinite index, quasiconvex
subgroup of $G$. Then $\Lambda_G(H)$ is nowhere dense in $\pa G$.
\end{lemma}
\begin{proof}
Without loss of generality both $G$ and $H$ are infinite groups. We prove the contrapositive statement.
Suppose $U\subset \Lambda_G(H)$ is a nonempty open subset of $\pa G$. Then
the set $V=\bigcup_{h\in H}h.U$ is an open and $H$-invariant subset of $\Lambda_G(H)$.
Also $V$ is an open subset of $\pa G$. Thus $\Lambda_G(H)\setminus V$ is an $H$-invariant
closed subset of $\Lambda_G(H)$. As $\Lambda_G(H)$ is the minimal, closed, $H$-invariant
subset of $\pa G$ (Lemma \ref{limit set minimal} $(2)$), and $U\neq \emptyset$ it follows
that $V=\Lambda_G(H)$ and it is an open subset of $\pa G$.

Now, as $H$ is an infinite hyperbolic group there is an infinite order element, say, $h\in H$.
Let $\xi \in \pa G$ be any element other than $h^{\pm \infty}$. Such a $\xi$ exists by
Proposition \ref{qc prop}(10) for otherwise $H$ will be of finite index in $G$. Now,
$h^n.\xi\map h^{\infty}$ as $n \map \infty$ by Lemma \ref{char loxo isom} $(2)$.
As $V$ is an open neighbourhood of $h^{\infty}$ in $\pa G$,
for all large $n$, we have $h^n.\xi \in V$. As $V$ is $H$-invariant it follows that $\xi\in V$.
Thus $V=\pa G$. Hence, $H$ has finite index in $G$ by Proposition \ref{qc prop}(10).
\end{proof}

As the topology on $\pa G$ makes it into a compact metrizable space, the Baire category theorem
gives the following immediate corollary.

\begin{cor}\label{cor-baire category}
Let $m\in\N$. Suppose $G$ is a hyperbolic group and $\{H_1, H_2, \cdots, H_m\}$ is a finite number of infinite index
quasiconvex subgroups of $G$. Then for any finite subset $S\sse G$, the set $\bigcup _{g\in S,~ 1\leq i\leq m} \Lambda_G(g H_i)$
is nowhere dense in $\pa G$. In particular, it is a proper subset of $\pa G$.
\end{cor}

The following result easily follows from \cite[lemma $1.2$]{GMRS}. However, for completeness we give a different proof. We refer the reader to Section \ref{sec-gen on mal} for relevant notation used in Lemma \ref{egh finite}.
\begin{lemma}\label{egh finite}
Suppose $G$ is a hyperbolic group and $H$ is a quasiconvex subgroup of $G$.
Then the set $\mathcal E_G(H)=\{H\cap H^g:g\in\mathcal I_G(H)\}$ has finitely many subgroups of $H$ up to conjugation
by elements of $H$.
\end{lemma}

\begin{proof} Let $\Gamma_G$ be a Cayley graph of $G$ with respect to a finite generating set.
Suppose $\Gamma_G$ is $\delta$-hyperbolic and $H$ is a $k$-quasiconvex in $\Gamma_G$.
Let $g\in \mathcal I_G(H)$ and $E=H\cap H^g$. Then by Proposition \ref{qc prop}(8), $E$ is an
infinite quasiconvex subgroup of $G$. Then the 1st part of the same proposition implies
that $E$ is an infinite hyperbolic group. It follows that $\Lambda_G(E)$ has at least two
distinct points by Proposition \ref{qc prop}(3).
Let $\xi_1\neq \xi_2$ be points of $\Lambda_G(E)$ and let $\gamma$ be a geodesic
line in $\Gamma_G$ joining them. Then there is a point $h\in H$ such that
$d_G(\gamma, h)\leq D=D(\delta,k)$ where is $D$ is as in Proposition \ref{qc prop}(9).
Hence, $d_G(1, h^{-1}\gamma)\leq D$. We note that (1)
$h^{-1}\xi_1,~h^{-1}\xi_2\in \Lambda_G(h^{-1}E)=\Lambda_G(h^{-1}Eh)$;
and (2) also $h^{-1}\xi_1,~h^{-1}\xi_2\in \Lambda_G (h.gHg^{-1})=\Lambda_G(h^{-1}gH)$.
As $h^{-1}gH$ is also a $k$-quasiconvex subset of $\Gamma_G$, we have
$d_G(1, h^{-1}gH)\leq 2D$. Let $x=h^{-1}g$. We note that the number of such left cosets $xH$
of $H$ such that $d_G(1, xH)\leq 2D$ is finite. Hence, $h^{-1}Eh=H\cap xHx^{-1}$ is
also finite. The Lemma follows. 
\end{proof}

\begin{lemma}\label{qc conjugate}
	Suppose $G$ is a hyperbolic group and $H$ is a quasiconvex subgroup.
	Let $x,y\in G$, and let $x$ be of infinite order. Then $x^n\in yHy^{-1}$ for some $n\in \N$ if and only if $x^{\infty}\in \Lambda_G(yH)$.
\end{lemma}
\begin{proof}
	We note that $yHy^{-1}$ is a quasiconvex subgroup of $G$ and so is $\lgl x\rgl$. Now, the
	lemma follows from the limit set intersection property of quasiconvex subgroups of
	hyperbolic groups, i.e. Proposition \ref{qc prop}(8) and the fact that
	$\Lambda_G(yHy^{-1})=\Lambda_G(yH)$, as follows.

	If $x^n\in yHy^{-1}$ then clearly $x^{\infty}\in \Lambda_G(yHy^{-1})=\Lambda_G(yH)$.
	On the other hand if, $x^{\infty}\in \Lambda_G(yH)$, then we have
	$x^{\infty}\in \Lambda_G(\lgl x\rgl)\cap \Lambda_G(yHy^{-1})=\Lambda_G(\lgl x\rgl\cap yHy^{-1})$.
	It follows that $\lgl x\rgl\cap yHy^{-1}=\lgl x^n\rgl$ for some $n\in \mathbb N$.
\end{proof}

\begin{lemma}\label{commen lemma}
Suppose $G$ is a hyperbolic group and $H$ is an infinite, quasiconvex subgroup of $G$.
Let $g\in  G\setminus Comm_G(H)$. Then $H\cap H^g$ has infinite index in $H$.
\end{lemma}
\proof We prove the contrapositive statement.
Suppose $x\in G$ such that $H\cap H^x$ has finite index in $H$.

{\bf Claim:} $\Lambda(H)=x\Lambda(H)$.

By Lemma \ref{limit set basic}(3)(i), $\Lambda(H)=\Lambda(H\cap H^x)$. By the limit set intersection
property (Proposition \ref{qc prop}(8)), $\Lambda(H\cap H^x)=\Lambda(H)\cap \Lambda(H^x)$
as $H^x$ is a quasiconvex subgroup of $G$ too. Also by Lemma \ref{limit set basic}(3)(ii),
$\Lambda(H^x)=\Lambda (xH)=x\Lambda(H)$. Thus we have $\Lambda(H)=\Lambda(H)\cap x\Lambda(H)$.
This means that $\Lambda(H)\subset x\Lambda(H)$ whence $\Lambda(H)\supset x^{-1}\Lambda(H)$.
Then inductively, by repeated application of $x^{-1}$, we have a descending chain
$$\Lambda(H)\supset x^{-1}\Lambda(H)\supset x^{-2}\Lambda(H)\supset \cdots$$
Let $Z=\cap_{n\geq 0} x^{-n}\Lambda(H)$. It is easy to check that $Z$ is a
closed subset of $\pa G$ which is invariant under $\lgl x\rgl$. Now there are two cases to consider

{\bf Case 1.} Suppose $x$ has finite order. Let $x^n=1$. In that case,
$\Lambda(H)=x^{-n}\Lambda(H)$. Hence, from the descending chain it follows that
$\Lambda(H)=x^{-1}\Lambda(H)$, and thus $x\Lambda(H)=\Lambda(H)$.

{\bf Case 2.} Suppose $x$ has infinite order.
In this case, as each $x^{-n}\Lambda(H)$ is a nonempty closed subset of $\pa G$ which is
a compact space. Hence, by the finite intersection property of compact spaces, $Z$ is a nonempty set.
Hence, by Lemma \ref{limit set minimal} $(2)$, $x^{\pm \infty}\in Z$.
In particular, $x^{\pm \infty}\in \Lambda(H)$.
Hence, by Lemma \ref{qc conjugate}, $x^n\in H$. Thus, once again, using the descending chain,
as $x^{-n}\Lambda(H)=\Lambda(x^{-n} H)=\Lambda(H)$, we have $\Lambda(H)=x^{-1}\Lambda(H)$,
and so $x\Lambda(H)=\Lambda(H)$. This completes the proof of the claim.

Now, using the claim we have $\Lambda(H)=x\Lambda(H)=\Lambda(xH)=\Lambda(H^x)$
as was observed in the first paragraph of the proof of the claim.
Finally, we can use Corollary \ref{qc sbgp commens} below to finish the proof.
\qed

\begin{lemma}\label{commens lemma 2}
Suppose $G$ is a hyperbolic group and $H, K$ are two infinite quasiconvex subgroups of $G$.
Suppose $K<H$ and $\Lambda(K)=\Lambda(H)$. Then $[H:K]<\infty$.
\end{lemma}
\proof Let $\xi, \xi'$ be any two distinct point of $\Lambda(H)=\Lambda(K)$. Let $\gamma$ be a
geodesic line in $\Gamma_G$ joining these points. Then we know by Proposition
\ref{qc prop}(9) that
$\gamma \subset N_D(K)\cap N_D(H)$ for $D\geq 0$ as $K, H$ are quasiconvex in $G$.
Let $x\in \gamma$ and $h\in H$ be such that $d_G(x,h)\leq D$. Then $h\in N_{2D}(K)$.
Now, for any $h'\in H$, $h'\gamma$ is a geodesic joining
$h'\xi, h'\xi'\in h\Lambda(H)=\Lambda(H)=\Lambda(K)$ as $\Lambda(H)$ is stable under the
action of $H$. Thus by the same argument $h'.x\in N_D(K)$. Now, $d_G(h'x, h'h)=d_H(x,h)\leq D$
and hence, $h'h\in N_{2D}(K)$. Thus $H\subset N_{2D}(K)$. From this it easily follows
that $[H:K]<\infty$.
\qed

\begin{cor}\label{qc sbgp commens}
Suppose $G$ is a hyperbolic group and $H,K$ are two quasiconvex subgroup of $G$.
If $\Lambda(H)=\Lambda(K)$ then $H$, $K$ are commensurable subgroups of $G$, i.e. $H\cap K$ has
finite index in both $H$ and $K$.
\end{cor}
\proof By limit set intersection property of quasiconvex subgroups, i.e. Proposition \ref{qc prop}
(8), we have that $H\cap K$ is quasiconvex in $G$ and $\Lambda(H\cap K)=\Lambda(H)=\Lambda(K)$.
Then we are done by Lemma \ref{commens lemma 2}.
\qed

\subsection{Coned-off spaces and applications}
We remind the reader that the coning construction or the electrocution was introduced by Farb
in \cite{farb-relhyp}. Suppose $Y$ is any graph and $\{A_i\}$ is a collection of subgraphs
in $Y$. Then the coned-off graph $\hat{Y}$ is the new graph obtained
from $Y$ by introducing new vertices and edges along with those already in $Y$ as follows:

(1) Vertices: For each set $A_i$ suppose we have a new vertex $v_i$.

(2) Edges: $v_i$ is joined by an edge to each vertex of $A_i$.

\begin{prop}\textup{(\cite[Proposition $2.6$]{KR-FFC})} \label{coned space hyp}
Suppose $G$ is a hyperbolic group and $\{H_i: 1\leq i\leq l\}$ are quasiconvex subgroups of $G$.
Suppose that in a Cayley graph $\Gamma_G$ we cone off all the cosets of all the subgroups $H_i$.
Then the coned-off graph $\hat{\Gamma}_G$ is hyperbolic.

\end{prop}

\begin{prop}\textup{(\cite[Corollary $6.4$]{ab-mn-acy}, \cite[Theorem $2.6$]{KR-FFC}),
(See also \cite[Section 3.2]{pranab-ravi-ct})} \label{geod in coned space}
Suppose $G$, $\Gamma_G$ and $\{H_i: 1\leq i\leq l\}$ are as in Proposition \ref{coned space hyp}.
Then any (quasi) geodesic ray $\gamma$ in $\Gamma_G$ satisfies exactly one of the following
two properties:

(1) $\gamma$ is a finite diameter subset of $\hat{\Gamma}_G$ which is equivalent to
$\gamma(\infty)\in \Lambda_G(gH_i)$ for some $g\in G$ and $1\leq i\leq l$.

(2) $\gamma$ is an infinite diameter subset of $\hat{\Gamma}_G$. In this case,
there is $\eta \in \pa \hat{G}$ such that $\gamma(n)\map \eta$ as $n\map \infty$.
Moreover, for any (quasi)geodesic ray $\alpha$ in $\hat{\Gamma}_G$ such that
$\alpha(\infty)=\eta$ one has $Hd(\alpha, \gamma)<\infty$.
\end{prop}

Let $\pa_u G\subset \pa G$ denote the set of points $\xi$ for which there is a geodesic
ray in $\Gamma_G$ which is of infinite diameter in $\hat{\Gamma}_G$. Then we have the following:

\begin{theorem} \textup{(\cite[Theorem 3.2]{dow-tay-cosurface} (see also \cite[Theorem $6.7$]{ab-mn-acy})}
\label{u-boundary}

The natural map $\pa_u G\map \pa \hat{G}$ given by Proposition \ref{geod in coned space}(2)
is a homeomorphism.
\end{theorem}

Suppose $G$ is a hyperbolic group and $\{H_i:1\le i\le l\}$ are quasiconvex subgroups of $G$.
Suppose $\Gamma_G$ is the Cayley graph of $G$ with respect to a finite generating set. Let $\hat{\Gamma}_G$
be the graph obtained from $\Gamma_G$ by coning off all the cosets of the subgroups $H_i$, $1\leq i\leq l$.
We note that the isometric action of $G$ on $\Gamma_G$ naturally extends to an action of $G$ on
$\hat{\Gamma}_G$ by isometries.

\begin{prop}\label{lem-bounded proj in hat}
Given any $g\in G$, its action on $\hat{\Gamma}_G$ satisfies exactly one of the following:

(1) $g$ induces an elliptic isometry on $\hat{\Gamma}_G$ i.e.
$diam\{\pi(g^n):n\in\Z\}$ is finite in $\hat{\Gamma}_G$. This happens
if and only if some power of $g$ is conjugate to an element of $H_i$ for some $i\in\{1,\cdots,l\}$.

(2) $g$ acts by a loxodromic isometry on $\hat{\Gamma}_G$.
In particular, in this case, $n\mapsto g^n$, $n\in \Z$ is a quasigeodesic line in $\hat{\Gamma}_G$.
\end{prop}

\begin{proof}
We shall denote the metric on $\Gamma_G$ by $d_G$ and the metric on $\hat{\Gamma}_G$ by $\hat{d}_G$.

(1) If $g\in G$ has finite order then clearly it acts as an elliptic isometry on $\hat{\Gamma}_G$.
Suppose $g\in G$ has infinite order.
Suppose $g^m\in xH_ix^{-1}$ for some $x\in G$, $m\in\Z\setminus\{0\}$ and $i\in\{1,\cdots,l\}$.
Since $Hd_G(xH_i,xH_ix^{-1})\le d_G(1,x)<\infty$ and $Hd(\lgl g\rgl, \lgl g^m\rgl)<\infty$, it follows that there is a constant $D$ depending on these Hausdorff
distances such that $\lgl g\rgl\subset N_D(xH_i)$ in $\Gamma_G$. As the inclusion $\Gamma_G\map \hat{\Gamma}_G$ is distance
decreasing, it follows that $\lgl g\rgl\subset N_D(xH_i)$ in $\hat{\Gamma}_G$ as well. But the diameter of $xH_i$ is $2$
in $\hat{\Gamma}_G$. Thus the diameter of the $\lgl g\rgl$ is finite in $\hat{\Gamma}_G$.

 Conversely, suppose $diam\{\pi(g^n):n\in\N\}$ is finite in $\hat{\Gamma}_G$.
We know by Lemma \ref{lem-infinite ele loxo} that $n\mapsto g^n$, $n\in \N$, defines a quasigeodesic ray in $\Gamma_G$.
Then by Proposition \ref{geod in coned space} $(1)$, $g^{\infty}\in \Lambda_G(xH_i)$
for some $x\in G$ and $1\leq i\leq l$. Then by Lemma \ref{qc conjugate}, $g^m\in xH_ix^{-1}$ for some $m\in \N$.

(2) Suppose $g\in G$ is an infinite order element that does not act as an elliptic isometry on
$\hat{\Gamma}_G$. It follows from (1) that $\alpha_{\pm}:n\mapsto g^n$ for $n\in \Z_{\geq 0}$ (respectively
$n\in \Z_{\leq 0}$ defines two quasigeodesic rays in $\Gamma_G$ that do not converge to the limit set
of any coset of any of the $H_i$'s. By Proposition \ref{geod in coned space} $(2)$, the sequences
$\{g^n\}_{n\in ]N}$ and $\{g^{-n}\}_{n\in \N}$ are converging to some points of $\partial \hat{\Gamma}_G$.
However, these points are not equal to each other by Theorem \ref{u-boundary}. On the other hand, clearly
these points are fixed points of $g$. Hence, we are done by Lemma \ref{lemma lox}.
\end{proof}

\begin{prop}\label{pingpong criteria}{\em (Criterion for pingpong pair)}
Let $G$ be a nonelementary hyperbolic group. Let $\{H_i:1\le i\le l\}$ be a finite collection of infinite and quasiconvex subgroups of $G$.
Let $g_1,g_2\in G$ be two infinite order elements of $G$ which have disjoint fixed point sets in $\partial G$.
Further, suppose that for all $i\in\{1,\cdots,l\}$, no power of $g_j$, $j=1,2$, is conjugate in $G$ to an element of $H_i$.
	Then there is $m\in\N$ such that
	
\begin{enumerate}
\item $\lgl g^m_1,g^m_2\rgl< G$ is a free subgroup of rank $2$, and

\item $\langle g^m_1,g^m_2\rangle\cap g^{-1}H_ig=\{1\}$ for all $g\in G$ and for all $i\in\{1,\cdots,l\}$.
\end{enumerate}
\end{prop}

\begin{proof}
Let $\Gamma_G$ denote a Cayley graph of $G$ with respect to a finite generating set, and let
$\hat{\Gamma}_G$ denote the graph obtained from $\Gamma_G$ by coning off all the cosets of $\{H_i:1\le i\le l\}$
in $G$. It follows by Proposition \ref{lem-bounded proj in hat} $(2)$ that $g_1, g_2$ act by loxodromic
isometries on $\hat{\Gamma}_G$. By Theorem \ref{u-boundary} these elements are independent as
$g_1, g_2$ have no common fixed points in $\partial G$ and the fixed points of these isometries are
in $\partial_u G$. Then, by Lemma \ref{lem-finding free qc} for some $m\in \N$,
$g^m_1, g^m_2$ generate a free group, say $H$ and that the orbit map $H\map \hat{\Gamma}_G$
will be a qi embedding. It then follows that each element of $H\setminus \{1\}$ will act as a
loxodromic isometry on $\hat{\Gamma}_G$. Thus, once again, by Proposition \ref{lem-bounded proj in hat} $(1)$,
we have $H\cap gH_ig^{-1}=\{1\}$ for all $g\in G$ and $1\leq i\leq l$. This completes the proof of the
proposition.
\end{proof}

\subsection{An application of Proposition \ref{pingpong criteria} to free groups}

Now we obtain an application of the above criterion of pingpong pairs for free groups.
In this subsection $G$ is a free group with the standard free basis $S=\{x_i:1\le i\le m\}$
where $m\geq 2$. Let $\Gamma_G$ be the Cayley graph of $G$ with respect to $S$.
In this case, it is easy to see that the topology on $\pa G$ is the same as the
metric topology coming from the following metric.

\begin{defn}{\em (Metric on $\pa G$)}
Let $\alpha, \beta$ be two geodesic rays in $\Gamma_G$ starting from $1$
and let $\alpha(\infty)=\xi$, $\beta(\infty)=\eta$. Suppose that the segment
along which $\alpha, \beta$ overlap is of length $n$. Then we define
$d(\xi, \eta)=2^{-n}$
\end{defn}

The following lemma follows easily from the definition of the above metric.
\begin{lemma}\label{free lemma 1}
Let $w\in G$. Let $V_w\subset \pa G$ be the set of all $\xi\in \pa G$ such that
there is a geodesic ray $\alpha: [0,\infty)\map \Gamma_G$ with $\alpha(0)=1$,
$\alpha(\infty)=\xi$ and $[1,w]\subset \alpha$.

Then $V_w$ is a closed subset of $\pa G$ with nonempty interior.
\end{lemma}

We recall that given $g,h\in G$, we say that the product $gh$ is reduced if
there is no cancellation in the expression for $gh$ in terms of the basis $S$,
or equivalently, $d_G(1,gh)=d_G(1,g)+d_G(g,gh)$ where $d_G$ is the metric
on the Cayley graph $\Gamma_G$.

\begin{lemma}\label{lem-var-free-grp}
Let $w\in G\setminus [G,G]$.
Then the closure of the set
$$\A=\{(wg)^{\infty}:g\in [G,G], \mbox{ and}\,  wg \,
\text{cyclically reduced and reduced} \}$$ in $\pa G$ has nonempty interior
in $\pa G$.
\end{lemma}

\begin{proof}
It is enough to show that the closure of $\A$ in $\pa G$ is the set $V_w$ defined
in Lemma \ref{free lemma 1}. 
For this we first make the following observation.

We note that for any cyclically reduced element $g\in G$ concatenation of the
geodesic segments $[g^n,g^{n+1}]$, $n\geq 0$ gives a geodesic ray starting from
$1$. Let $\gamma_g$ denote this geodesic ray. If $\xi\in \pa G$ and $\alpha$
is a geodesic ray in $\Gamma_G$ joining $1$ to $\xi$ and $\{g_n\}$
is a sequence of cyclically reduced elements then $g_n\map \xi$ if and only if
$g^{\infty}_n=\gamma_{g_n}(\infty)\map \xi$ as $n\map \infty$.

Thus we have $\A\sse V_m$. Now, let $\alpha$ be a geodesic ray starting from $1$ and
ending at a point of $V_w$. Suppose $\alpha(n_0)=w$. Now, for any $n\in \N$ we define
below a reduced element $g_n$ of $[G,G]$ such that $wg_n$ is reduced and cyclically reduced.
Let $\iota: G\map G$ be the automorphism of $G$ given by $\iota (x_i)=x^{-1}_i$ for all
$1\leq i\leq m$. We note that for all $g\in G$, $g.\iota(g)\in [G,G]$.

There two cases to consider to define $g_n$'s.

{\bf Case 1.} Suppose $w$ starts and ends with nonzero powers of
the same element of $S$. Without loss of generality suppose $w$ starts
and ends with some nonzero powers of $x_1$. Now, depending on $t=w^{-1}\alpha(n_0+1)$
there are two subcases to consider.

{\bf Subcase 1.1.} Suppose $t=x_1$ or $t=x^{-1}_1$.
Now, for all $k\in \N$ we look at the word $w_k=w^{-1}\alpha(n_0+k)$.
We note that $w_k=x^{r_1}_{p_1}x^{r_2}_{p_2}\ldots x^{r_i}_{p_i}$
where  $1\leq p_j\leq m$, $p_1=1$ and $r_1, r_2,\cdots, r_i$ are all nonzero integers.
Suppose there is subsequence of natural numbers $\{k_n\}$
such that  for all $n\in \N$, $w_{k_n}$ ends with a power of $x_2$. Then we define
$g_n$ as follows: If $w_{k_n}=x^{r_1}_{1}x^{r_2}_{p_2}\ldots x^{r_i}_{2}$ then we let
$g_n=w_{k_n}.\iota(w_{k_n})$. Then clearly $wg_n$ is reduced and cyclically reduced and
$[1,wg_n]\cap \alpha$ contains $[1,ww_{k_n}]$ as a subsegment, whence $(wg_n)^{\infty}\in\A$ and $(wg_n)^{\infty}\map\al(\infty)$ as $n\map\infty$.

On the other hand if for all but finitely many $w_k$'s end with a nonzero power of
$x_1$ then we define  $g'_n=w_n.x_2.\iota(w_n.x_2)$ for all $n\in \N$.
We note that clearly $\{g_n\}$ can be taken to be a suitable subsequence
of $\{g'_n\}$.

{\bf Subcase 1.2.} Suppose $t$ is either $x_2$ or $x^{-1}_2$.
We define $w_k$'s as in Subcase 1.1 and look for a sequence of natural numbers
$\{k_n\}$ such that $w_{k_n}$ ends with a power of $x_1$. If such a sequence
exists then we define $g_n =w_{k_n}.t^{-1}.\iota(w_{k_n}t^{-1})$.

On the other hand, if all but finitely many $w_k$'s end with $x_2$
then for all large $n$ we define
$g_n=w_n.x_1.t^{-1}.\iota(w_n.x_1.t^{-1})$.

{\bf Case 2.} Suppose $w$ starts and ends with the power of two different symbols.
In this case also the choices of $g_n$'s can be made in a similar manner. However,
we skip the details to avoid repetition, and instead leave it as a simple
exercise for the reader.

This completes the proof of Lemma \ref{lem-var-free-grp}.
\end{proof}

The following result is similar to \cite[Theorem $5.16$]{kapovich-nonqc}.
\begin{prop}{\em (Existence of pingpong pair)} \label{pingpong prop}
Suppose $G$ is a free group of rank $\geq 2$ and $\{E_i\}_{1\leq i\leq l}$ are some finitely generated, infinite
index subgroups of $G$. Let $w\in G\setminus [G,G]$ and let
$A=\{wg: g\in [G,G], wg\, \mbox{reduced and cyclically reduced}\}$.
Then there are $g_1,g_2\in A$ and $k\in \N$ such that

(1) $\lgl g^k_1, g^k_2\rgl\isom \mathbb F_2$.

(2) $\lgl g^k_1, g^k_2\rgl\cap \, xE_ix^{-1}=\{1\}$ for all $x\in G$ and $1\leq i\leq l$.
\end{prop}
\begin{proof}
By Proposition \ref{pingpong criteria} it is enough to find two
elements $g_1,g_2\in A$ such that $\{ g^{\pm \infty}_1\}\cap \{ g^{\pm \infty}_2\}=
\emptyset$ and no power of either $g_1$ or $g_2$ belong to a conjugate of
any of the $E_i$'s. However, we make the following notes first. Suppose $x\in A$.

\smallskip
(1)  As $x$ is cyclically reduced, concatenation of the
geodesic segments $[x^m,x^{m+1}]$ gives the geodesic line, say $\gamma_x$, in
$\Gamma_G$ joining $x^{\pm \infty}$ which passes through $1$.

\smallskip
(2) By Lemma \ref{qc conjugate} some nonzero power of $x$ belongs to
$gE_ig^{-1}$ for some $g\in G$ and $1\leq i\leq l$ if and only if $x^{\pm\infty}\in \Lambda_G(gE_i)$.
However, each $\Lambda_G(gE_i)$ is nowhere dense in $\pa G$ by Lemma
\ref{prop-finite-union-nowheredense}. We note that the set of cosets
$\{gE_i: g\in G, 1\leq i\leq l, d_G(1,gE_i)\leq D\}=\mathcal F_D$, say, is finite for all $D\geq 0$.
Thus the set $\cup_{gE_i\in \mathcal F_D} \Lambda_G(gE_i)$ is also nowhere dense in $\partial G$
by Baire category theorem (Corollary \ref{cor-baire category}). Now, the closure of $\mathcal A=\{g^{\infty}: g\in A\}$
has nonempty interior in $\pa G$ by Lemma \ref{lem-var-free-grp}.
Thus the closure of $\A\setminus \cup_{gE_i\in \mathcal F_D} \Lambda_G(gE_i)$ has nonempty interior as well.
Finally, we note that any nonempty open subset of $\pa G$ has infinitely many points.
It follows that the set $\A\setminus \cup_{gE_i\in \mathcal F_D} \Lambda_G(gE_i)$ is infinite.
Hence, we may choose $g_1, g_2$ from $\A\setminus \cup_{gE_i\in \mathcal F_D} \Lambda_G(gE_i)$
satisfying the desired properties.
\end{proof}

\section{Proof of the main theorem}\label{sec-main thm}

We shall need the following final piece of results for the proof of the main theorem
(Theorem \ref{main malnormal thm}).
\begin{prop}\label{main prop}
Suppose $G$ is a finitely generated, virtually free group which is not virtually cyclic. Let $H$ be a normal, free
subgroup of $G$ of finite index. Suppose $\mathcal E$ is a finite set, possibly empty, of
subgroups of $H$ such that
each $E\in \mathcal E$ is infinite, finitely generated and of infinite index in $H$.

Then there is a free subgroup $H_1<H$ of rank at least $2$ such that
\begin{enumerate}
\item $H_1\cap h^{-1}Eh=\{1\}$ for all $h\in H$ and for all $E\in \mathcal E$,
\item There is a finite normal subgroup $A$ of $Comm_G(H_1)$
and a finitely generated free subgroup $\mathbb F$ such that $H_1<\mathbb F$ and
$Comm_G(H_1)=\mathbb F A$. Consequently, we have a natural isomorphism
$Comm_G(H_1)\isom \mathbb F\ltimes A$ where $\mathbb F$ acts on $A$ by conjugation.
\end{enumerate}
\end{prop}

\begin{proof}
{\bf Step 1.} As the first step of the proof, we indicate how to construct a free subgroup $H_1$
of $H$ such that $Comm_G(H_1)$ is of the form described in (2).

Since $H$ is normal in $G$, we may consider the action of $G$ on $H$ by conjugation.
This gives a homomorphism $\phi:G\map Aut(H)$. Clearly, $\phi$ restricted to $H$ is injective.
As $G/H$ is finite it follows that $Ker(\phi)= Z_{G}(H)$ is finite. Let $H^{ab}$ denote the abelianization of $H$.
Let $\bar{\phi}:G \map Aut(H^{ab})$ be the composition of $\phi$ and the natural map $Aut(H)\map Aut(H^{ab})$.
Note that $\bar{\phi}(H)$ is trivial. As $H$ has finite index in $G$, it follows that the
image, say $L$, of $\bar{\phi}$ is finite.

Let $\{x_1,x_2,\cdots, x_m\}$ be a free basis of $H$. Let $\pi:H\map H^{ab}$ be the natural
quotient map. Let $\{\pi(x_1)=\bar{x}_1,\pi(x_2)=\bar{x}_2,\cdots, \pi(x_n)=\bar{x}_m\}$
be the corresponding basis of the free abelian group $H^{ab}$. We note that by hypothesis $m\geq 2$.
We shall use additive notation for $H^{ab}$. Now, since $L$ is finite, we may choose $N\in\N$
such that $v=\pi(x_1x^N_2)=\bar{x}_1+N\bar{x}_2$ is not an eigen vector of any element of
$L\setminus \{1\}$.

\smallskip
{\bf Claim} $1$: {\em If $H_0<H$ is such that $\pi(H_0)=\lgl av\rgl$ for some $a\in \N$,
then $Z_G(H)$ is the normal subgroup of $Comm_G(H_0)$ which consists of all the
finite order elements of $Comm_G(H_0)$. }

{\em Proof of Claim 1}: First of all, we note that $Ker(\phi)=Z_G(H)<Comm_G(H_0)$ and all elements of
$Z_G(H)$ are of finite order as $Z_G(H)$ is finite noted above.

On the other hand, suppose $g\in Comm_G(H_0)$ is a finite order element, and suppose
$g\not \in Ker(\phi)=Z_G(H)$. Then by Theorem \ref{finite gp inj},
$\bar{\phi}(g)\neq 1$. Let $\bar{g}=\bar{\phi}(g)$ and $\bar{H}_0=\pi(H_0)$.
By the choice of $N$, $av$ is not an eigen vector of $\bar{g}$. Now, as $\pi$
is clearly equivariant under the action of $G$ on $H$ and $H^{ab}$ (through $\phi$
and $\bar{\phi}$ respectively), we have
$\pi(\phi(g).H_0)= \pi(gH_0g^{-1})=\lgl \bar{\phi}(g)(av)\rgl$.
It follows that $\pi(H_0)\cap \pi(gH_0g^{-1})=(0)$, whence
$g\not \in Comm_G(H_0)$ $-$ a contradiction. Hence, any finite order element of
$Comm_G(H_0)$ is contained in $Z_G(H)$. This completes the proof of Claim $1$.

\smallskip
{\bf Claim $2$}: {\em Suppose $H_0<H$ is as in Claim $1$. Then there is (i) a (finite index) free subgroup
$\mathbb F$ of $Comm_G(H_0)$ and (ii) a finite, normal
subgroup $A$ of $Comm_G(H_0)$ such that $Comm_G(H_0)=\mathbb F A$. In particular,
$Comm_G(H_1)=Comm_G(H_0)\isom \mathbb F \ltimes A$ where $H_1= H_0\cap \mathbb F$.}

{\em Proof of Claim 2}: Let $A=Z_G(H)$. We have already noted that it is a finite group.
From the proof of  Claim 1, as $A$ contains all the finite order elements of $Comm_G(H_0)$,
it follows that $A$ is a normal subgroup of $Comm_G(H_0)$. Thus
$Comm_G(H_0)/A$ is a finitely generated, torsion free, virtually free
group. This means that it is a free group by \cite[Theorem $0.2$]{stallings-freegp}.
Hence the quotient homomorphism
$\pi: Comm_G(H_0)\map Comm_G(H_0)/A$ admits splitting homomorphism
$\psi: Comm_G(H_0)/A \map Comm_G(H_0)$ so that $\pi \circ \psi$ is the identity
map on $Comm_G(H_0)/A$. Let $\mathbb F$ be the image of $\psi$. Then clearly
$\mathbb F\cap A=(1)$ and $Comm_G(H_0)= \mathbb F A$. It follows that
$Comm_G(H_0)\isom \mathbb F \ltimes A$.

Finally, we note that both $H_0$ and $\mathbb F$ are finite index subgroups of
$Comm_G(H_0)$. Hence, so is $H_1= H_0\cap \mathbb F$ and therefore, (by Lemma \ref{lemma: commensurator})
$Comm_G(H_1)=Comm_G(H_0)=\mathbb F A\isom \mathbb F\ltimes A$.

If $\mathcal E$ were empty we would be done with the proof of the proposition by Step 1. If not then
we will need to proceed to the next step.

{\bf Step 2.} In the second step of the proof we construct $H_1<H$ so that
$H_1$ satisfies both (1) and (2) mentioned in the proposition. As in Step 1, we will ensure that
$\pi(H_1)$ is an infinite cyclic subgroup of $\lgl v\rgl$ so that property (2)
will follow from Claim $2$. As in Step 1, we work with the free basis $S=\{x_1, x_2, \cdots, x_m\}$
of $H$.

Let $w=x_1x^N_2$ where $N$ is as in Step 1. Let
$$A=\{wh\in H:\,h\in[H,H]\text{ and }wh\text{ is cyclically reduced}\}$$.
Then by Proposition \ref{pingpong prop}, one can find $g_1,g_2\in A$ such that

\begin{enumerate}
\item $\lgl g^k_1,g^k_2\rgl<H$ is a free subgroup of rank $2$ for some $k\in \N$, and

\item $\lgl g^k_1,g^k_2\rgl\cap h^{-1}Eh=\{1\}$ for all $h\in H$ and for all $E\in \mathcal E$.
\end{enumerate}
We let $H_0=\lgl g^k_1, g^k_2\rgl$. Then $H_0$ satisfies property (1) of the proposition. Note that
by construction of $H_0$, the image of $H_0$ under the quotient map $H\map H^{ab}$ is
the infinite cyclic group $\lgl kv\rgl$. Then as in Claim 2 of  Step 1, $H_1<H_0$ can be chosen which
satisfies both properties (1) and (2) of the proposition.
\end{proof}

\begin{theorem}\label{main malnormal thm}
Let $G$ be a nonelementary hyperbolic group. Let $H$ be a nonelementary, quasiconvex
subgroup of $G$. Then there is a free subgroup $H_1\simeq \mathbb F_2$ of $H$ such that
the following hold:
\begin{enumerate}
\item $Comm_G(H_1)$ is of the form $H_1\times A$ for some finite subgroup $A$ of $G$. 

\item $H_1A\isom H_1\times A$ is weakly malnormal (and quasiconvex) in $G$.
\end{enumerate}
\end{theorem}

\begin{proof}

{\bf Step 1.} First we prove an analogue of Proposition \ref{main prop}
to get hold of a free quasiconvex subgroup whose commensurator in $G$
has a rather simple structure.

\begin{prop}\label{step 1 of thm}
If $G$ is a nonelementary hyperbolic group and $H<G$ is a nonelementary quasiconvex
subgroup then there is subgroup $F$ of $H$ such that with the following properties.
\begin{enumerate}
\item $F\isom \mathbb F_n$ where $n\geq 2$ and $F$ is quasiconvex in $G$.
\item $F$ is a (finite index) normal subgroup of $Comm_G(F)$.
\item $Comm_G(F)=\mathbb F A \isom \mathbb F \ltimes A$
where $A$ is a finite normal subgroup of $Comm_G(F)$ and $\mathbb F$
is a free group containing $F$ as a finite index subgroup.
\end{enumerate}
\end{prop}
\begin{proof}
The proof idea is to construct a descending sequence of finitely generated
free groups $H\supset F_1 \supset F_2\cdots $ such that we have better control
on $\mathcal I_G(F_i)$ at each stage (see Section \ref{sec-gen on mal} for the notation).

As $H$ is nonelementary, by Proposition \ref{qc prop}(6), we can find a free subgroup
$F_1$ of $H$ isomorphic to $\mathbb F_2$ which is quasiconvex in $G$. Let $G_1= Comm_G(F_1)$.
Then $F_1$ has finite index in $G_1$ by Proposition \ref{qc prop}(4). Hence, we can
find a finite index subgroup  $F_2$ of $F_1$ which is normal in $G_1$.
We note that $Comm_G(F_1)=Comm_G(F_2)$ (Lemma \ref{lemma: commensurator}). Thus $F_2$ is a finitely generated free
quasiconvex of $G$ such that $F_2$ is normal in $Comm_G(F_2)$.

By Lemma \ref{commen lemma}, for all $g\in G\setminus Comm_G(F_2)$, $F_2\cap F^g_2$ has
infinite index in $F_2$. We recall that $F_2$ acts on the set
$\mathcal E'_G(F_2)=\{E<F_2: E=F_2\cap F^g_2,\, g\in \mathcal I_G(F_2)\setminus Comm_G(F_2)\}$
by conjugation (Lemma \ref{lem-properties of I(H)}), and by Lemma \ref{egh finite}, there are finitely many $F_2$-orbits of subgroups
in $\mathcal E'_G(F_2)$. Let $\mathcal E\subset \mathcal E'_G(F_2)$ be a finite subset containing
exactly one element from each of the $F_2$-orbits of infinite subgroups of the form $F_2\cap F^g_2$, $g\in G$.
As $Comm_G(F_2)=G_2$, say, is virtually free, by Proposition \ref{main prop}, we may
construct a subgroup $F_3<F_2$ with the following two properties:\smallskip

(1) $F_3\cap E^h=\{1\}$ for all $E\in \mathcal E$ and for all $h\in F_2$.
We note that this implies 
$Comm_G(F_3)=Comm_{G_2}(F_3)$. Indeed, let $x\notin G_2$ such that $[F_3:F_3\cap F^x_3]<\infty$. In particular, $x\in \mathcal I_G(F_2)\setminus Comm_{G}(F_2)$. Thus $F_3\cap (F_2\cap F^x_2)$ is infinite $-$ contradicting to the fact that $F_3\cap E^h$ is trivial.\smallskip

(2) There is a free subgroup $\mathbb F<Comm_G(F_3)$ and a finite normal
subgroup $A$ of $Comm_G(F_3)=\mathbb F A\isom \mathbb F \ltimes A$.

Finally, we define $F=\mathbb F\cap F_3$. Clearly then $F$ satisfies the properties
of the proposition.
\end{proof}

{\bf Step 2.} We continue to use the notations and definition from Proposition \ref{step 1 of thm} and
its proof; e.g. let $F$ still denote the subgroup of $G$ as constructed in Proposition
\ref{step 1 of thm}.  Next, we seek a subgroup of $F$ with the desired properties.
The proof runs, once again, along the lines of Proposition \ref{step 1 of thm}.

Let $G_3=Comm_G(F)$.
Once again, let $\{E_1, \ldots, E_l\}$ be a finite set of representatives
of infinite and infinite index subgroups of $F$ of the form $F\cap F^g$ up to conjugation
by elements of $F$, where
$g\in G\setminus Comm_G(F)$. Then by Proposition \ref{pingpong prop}, there is a
free subgroup $F_4$, say, in $F$ of rank two such that
$F_4\cap E^h_i$ is finite for all $h\in F$ and $1\leq i\leq l$.
We note the inclusions $F_4<F<\mathbb F<G_3$. If necessary, by passing to a further
subgroup we may assume that $A< Z_G(F_4)$. Now, using Theorem \ref{malnormal in free},
one can find a subgroup of $F_4$ which is isomorphic to $\mathbb F_2$ and which is
malnormal in $\mathbb F$.
Let $H_1$ be such a subgroup. Now it follows that $H_1$ satisfies the properties of
Theorem \ref{main malnormal thm}.

Indeed, $H_1A\isom H_1\times A$ is weakly malnormal in $G_3$, and so $Comm_{G_3}(H_1)\isom H_1\times A$. It follows, as explained in $(1)$ above (in the proof of Proposition \ref{step 1 of thm}), that $Comm_G(H_1)=Comm_{G_3}(H_1)=Comm_{G_3}(H_1A)\isom H_1\times A$. Finally, we have $(a)$ $H_1A\cap E_i$ is finite for all $g\in \mathcal I_G(F)\setminus G_3$, $(b)$ $H_1A$ has finite index in $Comm_{G_3}(H_1A)$, and $(c)$ $H_1A$ is weakly malnormal in $G_3$. It follows that $H_1A\isom H_1\times A$ is weakly malnormal in $G$.
\end{proof}

\begin{remark}\label{rmk-increasing rank}
 To increase the rank of $H_1$ in Theorem \ref{main malnormal thm}, we can take a malnormal subgroup $H_2$
 of any finite rank in $H_1$. (See \cite{wise-mal-1, wise-mal-2, das-mj-FS} for construction
 of malnormal subgroups of free groups.) Then $H_2\times A$ will be weakly malnormal in $G$ by
 Lemma \ref{malnormal transitive}.
\end{remark}

Following the proof \cite[Theorem $A$]{kapovich-nonqc} we immediately have the following application
of Theorem \ref{main malnormal thm}.
\begin{theorem}\label{main application}
Given any nonelementary hyperbolic group $G'$ there is another hyperbolic group
$G$ and an injective homomorphism $\phi:G'\map G$ such that
$\phi(G')$ is nonquasiconvex in $G$.
\end{theorem}
{\em Sketch of proof.} By Theorem \ref{main malnormal thm} (and Remark \ref{rmk-increasing rank}) we can find
a weakly malnormal quasiconvex subgroup $H$ of $G'$ isomorphic to
$\mathbb F_n\times A$ where $n\geq 3$ and $A$ is finite. Now, one may
choose a hyperbolic automorphism $f:\mathbb F_n\map \mathbb F_n$ (see \cite{BF}, \cite{BFH-lam}, \cite{brinkmann-auto}). (Note that $\psi=f\times 1$ is a hyperbolic automorphism of $H$.) Then the HNN extension $G=G'*_{\psi}$ can also be written as the free product with amalgamation $G=G'*_{H}(H*_{\psi})$. Note that $H$ is weakly malnormal in $G'$. Then it follows from an easy case of
the Bestvina--Feighn combination theorem (\cite[Theorem $1.2$]{BF-Adn}) that $G$ is hyperbolic. For a different proof, see for instance, \cite[Theorem $3.5$ $(1)$]{HMS-landing}, \cite{mats-oguni}. Now we will see that $G'$ is not quasiconvex in $G$. Since $H$ is quasiconvex in $G'$, it is enough to show that $H$ is not quasiconvex in $G$. 
Note that $H$ in $H*_{\psi}$ is not quasiconvex (\cite{BFH-lam}, \cite{mitra-endlam}). It follows that $H$ is not quasiconvex in $G$.


\smallskip
{\bf Acknowledgement.} RH was supported by the Visiting
(Postdoctoral) Fellowship from TIFR Mumbai, India.

\bibliography{Ubib}
\bibliographystyle{amsalpha}
\end{document}